\documentclass{amsart}

\usepackage[dvips]{graphicx}
\usepackage{color}
\usepackage{psfrag}
\usepackage{caption}

\newtheorem{thm}{Theorem}[section]
\newtheorem{lem}[thm]{Lemma}
\newtheorem{prop}[thm]{Proposition}
\newtheorem{cor}[thm]{Corollary}

\newtheorem{question}{Question}

\theoremstyle{definition}
\newtheorem{dfn}[thm]{Definition}
\newtheorem{rem}[thm]{Remark}
\newtheorem{example}[thm]{Example}

\def\hpic #1 #2 {\mbox{$\begin{array}[c]{l} \epsfig{file=#1,height=#2} \end{array}$}}
\def\vpic #1 #2 {\mbox{$\begin{array}[c]{l} \epsfig{file=#1,width=#2} \end{array}$}}
\def\ignore #1 {}

\def\R{\mbox{{$\mathbb{R}$}}}

\def\N{\mbox{{$\mathbb{N}$}}}

\def\A{\mbox{$\forall$}}

\renewcommand{\d}{\delta}
\newcommand{\g}{\gamma}

\renewcommand{\A}{{\mathcal A}}
\newcommand{\C}{{\mathcal C}}
\newcommand{\e}{\epsilon}
\newcommand{\s}{\sigma}
\renewcommand{\t}{\tau}
\newcommand{\Sym}{S}
\newcommand{\D}{\Delta}
\newcommand{\Leb}{\mu_{\text{Leb}}}

\newcommand{\bdry}{\partial}

\def \aaron #1 {\marginpar{#1 -AA}}
\newcommand{\AMP}{{\mathcal {A}}^{\rm{mp}}}
\newcommand{\bmu}{\mbox{\boldmath $\mu$}}
\newcommand{\bmun}{\mbox{\boldmath $\mu$}_n}
\newcommand{\bs}{\mbox{\boldmath $\sigma$}}

\DeclareMathOperator{\Path}{Path}
\DeclareMathOperator{\Int}{Int}
\DeclareMathOperator{\Drift}{Drift}
\DeclareMathOperator{\Order}{Order}
\DeclareMathOperator{\Max}{Max}
\DeclareMathOperator{\Min}{Min}

\begin{document}
   \DeclareGraphicsExtensions{.pdf,.gif,.jpg}

\title{Distributions of order patterns of interval maps}
\author[ABLLP]{Aaron Abrams, Eric Babson, Henry Landau, Zeph Landau, James Pommersheim}

\begin{abstract}
A permutation $\s$ describing the relative orders of the first $n$ iterates of a point $x$ 
under a self-map $f$ of the interval $I=[0,1]$ is called an \emph{order pattern}.
For fixed $f$ and $n$, measuring the points $x\in I$ (according to Lebesgue measure) that 
generate the order pattern $\s$ gives a probability distribution $\bmu_n(f)$ on the set 
of length $n$ permutations.  We study the distributions that arise this way for various 
classes of functions $f$.  

Our main results treat the class of measure preserving functions.
We obtain an exact description of the set of realizable distributions in this case:
for each $n$ this set is a union of open faces of the polytope of flows on a certain 
digraph, and a simple combinatorial criterion determines which faces are included.
We also show that for general $f$, apart from an obvious compatibility condition, 
there is no restriction on the sequence $\{\bmu_n(f)\}_{n=1,2,\ldots}$.

In addition, we give a necessary condition for $f$ to have \emph{finite exclusion type},
i.e., for there to be finitely many order patterns that generate all order patterns not realized 
by $f$.  Using entropy we show that if $f$ is piecewise continuous, 
piecewise monotone, and either ergodic or with points of arbitrarily high period, then $f$
cannot have finite exclusion type.  This generalizes results of S.~Elizalde. 
\end{abstract}

\maketitle

Given a function $f: [0,1] \to [0,1]$, it is natural to examine properties of the sequence of iterates 
of $f$ beginning at some point $x\in [0,1]$: 
$$x, f(x), f^2(x) \dots.$$
The \emph{order pattern} for a sequence of distinct reals $y_1,y_2,\dots,y_n$ is the 
permutation $\s \in S_n$ that ranks the elements in increasing order; specifically,
$y_i<y_j$ if and only if $\s(i)<\s(j)$.    A number of authors 
have explored the relationship between functions $f$ and the set of order patterns 
realized by the iterates of $f$.   Work of C.~Bandt, G.~Keller, B.~Pompe, J.~M.~Amig\'o, 
M.~Kennel, and M.~Misiurewicz \cite{bkp,bp,ak,m} relates the number of distinct order 
patterns arising from a function $f$ to the entropy of $f$.
S.~Elizalde and others \cite{sergi1,sergi2,sergi3} have examined which and how many 
order patterns do not arise for particular functions and classes of functions. 

Here we take a slightly broader view and investigate the collection of distributions 
of order patterns that particular classes of functions achieve.  Specifically, if $I=[0,1]$ is equipped 
with Lebesgue measure and $f$ is almost aperiodic (meaning that the set of points with 
finite orbit has measure zero) then $f$ induces a probability distribution
$\bmu_n(f)$ on $S_n$ in a natural way:
$$\bmu_n(f)(\s)=\Leb\{x \mid \Order(x,f(x),\dots,f^{n-1}(x))=\s\}.$$
We shall focus on the functions $\bmu_n$ as well as the function $\bmu$
which maps $f$ to the sequence $(\bmu_1(f),\bmu_2(f),\dots)$.

Throughout the paper we consider functions with the property that 
almost all orbits are infinite:
$$\A=\{f:I\to I \mid \Leb(I_{ap})=1\}$$
where $I_{ap}$ is the set of aperiodic points, i.e., points with infinite orbit.
We address the following natural questions:  if $\C\subset \A$ is a collection of functions, then
\begin{question}
What is $\bmun(\C)$?
\end{question}
\begin{question}
What is $\bmu(\C)$?
\end{question}

We begin by answering both questions for the class $\C = \A$. 
For any $f$, the distributions $\bmu_n(f)$, $n=2,3, \dots$ must satisfy a certain
compatibility condition.   In Theorem \ref{cantor} we show that this is the only constraint 
on what is realizable for arbitrary $f\in\A$:  that is, for any sequence $\{\mu_n\}_{n\geq 1}$ 
of compatible distributions on $S_n$, there is a function $f\in\A$ which simultaneously 
satisfies $\bmun(f)=\mu_n$.

We then turn our attention to the class of measure preserving functions, 
$$\C=\AMP=\{f\in\A \mid f \mbox{ preserves } \Leb\}.$$  
Our main theorem (Theorem \ref{main}) provides a complete answer to Question 1
for $\C=\AMP$.  It is easy to see that the conclusion of Theorem \ref{cantor} cannot
hold for $\C=\AMP$; in fact we observe that $\bmu_n(\AMP)$ is contained in a polytope
$P_n$ consisting of all (normalized) flows on a certain digraph, which we call a
\emph{permutation digraph}.  We then show that $\bmu_n(\AMP)$ is a union of open 
faces of $P_n$ including the top-dimensional face, and we give a combinatorial criterion 
for determining whether or not a given open face of $P_n$ is contained in $\bmu_n(\AMP)$.

To prove the main theorem we introduce the fundamental notion of \emph{drift}.
Naively, if one wants to construct $f\in\AMP$ realizing a given distribution $\mu$,
one might chop the interval into several subintervals and define $f$ to permute
the intervals to produce the desired frequencies.  Problems soon arise, however:
for example if we want half the mass of the interval to have iterates with order pattern 
$(1 3 2)$ and the other half $(2 1 3)$ then we quickly realize that this is impossible,
because $f^2$ would move all the mass to the right, which is impossible for a
measure preserving function.  This is the essence of drift, and
the upshot of Theorem \ref{main} is that this is the only obstruction:  
a face of $P_n$ either has drift or not, and the faces contained in $\bmu_n(\AMP)$ 
are exactly those without drift.

Finally, we discuss the relationship between the entropy of $f$ and a property
we call \emph{finite exclusion type}.  The latter is equivalent to $f$ having
finitely many \emph{basic forbidden patterns}, in the language introduced
by Amig\'o, Elizalde, and Kennel \cite{sergi1}; these properties mean that there are 
finitely many fixed patterns such that every permutation either arises as an order 
pattern of iterates of $f$ or contains one of the forbidden patterns.  A function with
finite entropy can realize at most exponentially many permutations of length $N$
(see \cite{bkp}), but using the notion of drift we show that quite often, a function
with finite exclusion type must realize a super-exponential number of permutations.
In particular, if either $f$ is continuous and has points of arbitrarily large period or 
$f$ is ergodic, then $f$ cannot have both finite entropy and finite exclusion type;
see Corollary \ref{entropyresult}.  This generalizes results from \cite{sergi2}.

The paper is organized as follows.  We introduce some language and give our result
for $\C=\A$ in Section 1, although we defer the proof to Section \ref{cantorproof}.
Sections 2-4 develop the combinatorial ideas required for our main theorem,
including several preliminary results about permutation digraphs and drift.  The main theorem
is stated and proved in Section \ref{mainsec}.  Our discussion of entropy and
finite exclusion type makes up Section 6, and Section \ref{cantorproof} contains
the proof of Theorem \ref{cantor}.  We close with some open questions in Section 8.

\section*{Acknowledgments}

The authors thank Sergi Elizalde for suggesting this line of research and for
several useful conversations along the way.  Thanks also to Julie Landau for her 
hospitality and kick serve.

\section{Generalities}

In this section we introduce some language and notation which will be used throughout the
paper, and we state our first result, Theorem \ref{cantor}, which says that if no
restriction is placed on $f$, then one can always find $f$ realizing a given
compatible sequence of permutation distributions.

\subsection*{Order patterns}
For a positive integer $n$ we denote $\{1,\dots,n\}$ by $[n]$ and
the group of bijections of $[n]$ by $S_n$.

Let $g$ be an injective map from a finite totally ordered set $X=\{x_1,\ldots,x_n\}$ 
(where $x_1<x_2<\cdots<x_n$) to a totally ordered set $Y$.  Let $y_i=g(x_i)$.
We define the \emph{order pattern} $\Order(g)$ to be the unique permutation 
$\s\in S_n$ satisfying $y_i<y_j$ if and only if $\s(i)<\s(j)$.  Equivalently
$y_{\s^{-1}(1)}<\cdots<y_{\s^{-1}(n)}$.  Note that if $\s\in S_n$ then $\Order(\s)=\s$.
The \emph{order pattern} of an $n$-tuple of distinct real numbers $(x_1,\dots,x_n)$
is $\Order(x_1,\dots,x_n)=\Order(g)$ where $g:[n]\to\R$ takes $i$ to $x_i$.

There is a restriction map $\rho:S_{n+1}\to S_n$ given by $\rho(\s)=\Order (\s|_{[n]})$.
Using this we define $S_\infty$ as 
$\{(\s_1,\s_2,\dots) \mid \s_i\in S_i, \rho(\s_{i+1})=\s_i \ \forall i\geq1\}$,
which is equal to the inverse limit of the maps $\rho:S_{n+1}\to S_n$.
Let $S=\cup_{n=1}^\infty S_n$.  The set $S$ is graded by $n$ and we use notation like
$S_{\geq n}$ to mean $\cup_{j=n}^\infty S_j$.

\subsection*{Distributions}
Next, let $\D_n$ be the space of probability distributions on $S_n$.  Note that $\D_n$ is the 
standard simplex in $\R^{S_n} \cong \R^{n!}$.  We denote by $\chi_\s$ the vertex of $\D_n$
which has mass 1 at $\s\in S_n$ and 0 elsewhere.

If $\mu\in\D_n$ and $\mu'\in\D_{n+1}$ we say $\mu$ and $\mu'$ are \emph{compatible}
if $\mu(\s)=\sum\limits_{\rho(\s')=\s} \mu'(\s')$.  

Then $\D_\infty=\{(\mu_1,\mu_2,\dots) \mid \mu_i\in\D_i, \mu_{i+1}\mbox{ and }\mu_i
\mbox{ are compatible } \forall i\geq 1 \}$, and $\D=\cup_{n=1}^\infty \D_n$.
As an example, the uniform distributions from each $\D_n$ form a compatible
sequence, hence an element of $\D_\infty$.

\subsection*{Induced distributions}
For $f\in\A$ and $x\in I_{ap}$ let $\bs^f(x)=(\bs_1^f(x),\bs_2^f(x),\dots)\in S_\infty$ where
$$\bs_n^f(x)=\Order(x,f(x),\dots,f^{n-1}(x)).$$
Let $\bmu_n:\A\to\D_n$ be the map taking a function $f\in\A$ to the distribution
defined by $$\bmu_n(f)(\s)=\Leb\{x \mid \bs_n^f(x)=\s\}.$$
Note that for any $f$ and $n$, the distributions $\bmu_n(f)$ and $\bmu_{n+1}(f)$ are 
compatible; thus we may define
$\bmu:\A\to \D_\infty$ by $\bmu(f)= (\bmu_1(f), \bmu_2(f),\dots)\in\D_\infty$.

We can now state our first result.

\begin{thm} \label{cantor}
For every $\mu = (\mu _1, \mu_2, \dots ) \in \D _{\infty}$ there exists a 
function $f\in\A$ with $\bmu(f)=\mu$.  That is, $\bmu(\A)=\D_\infty$.
\label{lem:cantor}
\end{thm}

The proof is constructive, a little involved, and unnecessary for the results
that follow.  Therefore we defer the proof to Section \ref{cantorproof}.

\subsection*{Convexity}
Before we end this section we make an observation about convexity.
Suppose $\C\subset\A$ is a collection of functions such that whenever $f,g\in \C$ and $t\in[0,1]$,
the function $$h(x)=\begin{cases}
tf(\frac x t) & \mbox{if } x<t \\
t+(1-t)g(\frac{x-t}{1-t}) & \mbox{if }t<x\leq 1
\end{cases}
$$
is also in $\C$.  Then $\bmu_n(\C)$ is a convex subset of $\D_n$.
This is because $h$ is the ``block sum" of $f$ scaled by $t$ and $g$ scaled by $1-t$, and
so for all $n$, $\bmu_n(h)=t\bmu_n(f)+(1-t)\bmu_n(g)$.

This will usually hold if $\C$ has ``piecewise'' in the title, such as piecewise continuous 
functions, piecewise monotone functions, etc.  It also holds for (aperiodic) measure 
preserving functions.

%


\section{Digraphs}

The next several sections develop the language used in the remainder of the paper.  
We begin with digraphs.

A \emph{digraph} is a quadruple $G=(VG,EG,h,t)$ with $VG$ the vertex set, $EG$ the edge
set, and $h$ and $t$ the head and tail maps from $EG$ to $VG$.  

Recall that $\rho:S_{n+1}\to S_n$ is defined by $\rho(\s)=\Order(\s|_{[n]})$.  Similarly 
define $\rho':S_{n+1}\to S_n$ by $\rho'(\s)=\Order(\s|_{[2,n+1]})$.

\begin{dfn}
For $n\geq 1$ let $G_n$ denote the \emph{permutation digraph} 
$(S_n,S_{n+1},\rho,\rho')$.  The digraphs $G_2$ and $G_3$ are shown
in Figure \ref{fig:digraphs}.
\end{dfn}

\begin{figure}[ht]
  \psfrag{123}{$123$}
  \psfrag{132}{$132$}
  \psfrag{213}{$213$}
  \psfrag{231}{$231$}
  \psfrag{312}{$312$}
  \psfrag{321}{$321$}
  \psfrag{12}{$12$}
  \psfrag{21}{$21$}  
  \centering
  \includegraphics[width=4.5in]{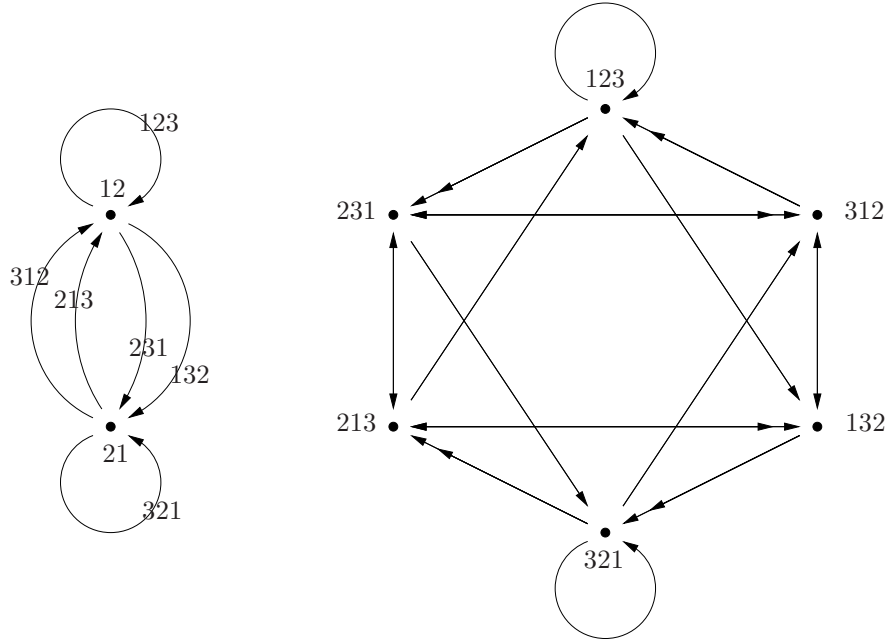}
  \captionsetup{width=.75\textwidth}
  \caption{The digraphs $G_2$ and $G_3$.  The edges of $G_2$ are shown with labels;
  the edges of $G_3$ are abbreviated and the labels omitted.  For instance two directed
  edges go from $231$ to $312$, with labels $2413$ and $3412$.  An edge labeled $4231$
  goes in the reverse direction.}
  \label{fig:digraphs}
\end{figure}

\subsection*{Paths}
A \emph{path} of length
$\ell$ (where $0\leq \ell < \infty$) in a digraph is an alternating sequence 
$p=(v_0,e_1,v_1,e_2,\ldots,v_\ell)$
with $v_i\in VG$ and $e_i\in EG$ such that $h(e_i)=v_{i-1}$ and $t(e_i)=v_i$.  A path
of length $\infty$ is $p=(v_0,e_1,v_1,\ldots)$ such that each finite initial segment ending
with a vertex is a (finite) path.  
We write $\Path_\ell(G)$ for the set of all paths of length $\ell$ in $G$ and $\Path(G)$ for
the set of all paths in $G$.  Note that $S_n=\Path_0(G_n)$.  To define specific paths we 
sometimes abuse notation slightly by thinking of $v_i$ and $e_i$ as functions from 
$\Path_{\geq i}(G)$ to $VG$ and $EG$.

For example, if $p$ is a path of finite length $\ell$ and $q$ is any path with $v_0(q)=v_\ell(p)$ 
then the \emph{concatenation} $pq$ of $p$ and $q$ has $v_i(pq)=v_i(p)$ and $e_i(pq)=e_i(p)$ 
for $i\leq\ell$ and $v_i(pq)=v_{i-\ell}(q)$ and $e_i(pq)=e_{i-\ell}(q)$ for $i>\ell$.

A digraph is \emph{strongly connected} if there are paths connecting every ordered pair
of vertices.  A (finite) path is \emph{embedded} if all $\ell+1$ vertices are distinct, except possibly 
$v_0=v_\ell$.  A \emph{loop} is a finite path with $v_0=v_\ell$.

\subsection*{Projections}
For each $n<\infty$ we define 
$$\pi_n:S_\infty \cup \left(\bigcup_{m\geq n} \Path(G_m)\right) \to \Path(G_n)$$ 
as follows.  First, $\pi_n$ is the identity on $\Path(G_n)$.
if $\s\in \Path_0(G_{n+1})=S_{n+1}$, let $\pi_{n+1,n}(\s)$ be the path
$(\rho(v),v,\rho'(v))\in\Path_1(G_n)$.  If $p=(v_0,e_0,\dots,v_\ell)\in\Path_\ell(G_{n+1})$ 
then let $\pi_{n+1,n}(p)$ be the concatenation 
$\pi_{n+1,n}(v_0)\pi_{n+1,n}(v_1)\cdots\pi_{n+1,n}(v_\ell)$.  (The result is an infinite path
if $\ell=\infty$; otherwise the result is a path of length $\ell+1$.)  Thus 
$\pi_{n+1,n}:\Path(G_{n+1})\to\Path(G_n)$.
Let $\pi_{m,n}=\pi_{n+1,n}\circ\cdots\circ\pi_{m,m-1}:\Path(G_m)\to\Path(G_n)$ and let
$\pi_n$ be the union of the functions $\pi_{m,n}$ on $\bigcup_{m\geq n} \Path(G_m)$.

Finally, extend $\pi_n$ further by defining $\pi_n(\s)$ for $\s=(\s_1,\s_2,\dots)\in S_\infty$
to be the infinite path whose initial subpath of length $\ell$ is equal to $\pi_n(\s_{\ell+n})$.

Note:  if $p\in\Path_\ell(G_m)$ then $\pi_n(p)\in\Path_{\ell+m-n}(G_n)$.

\subsection*{Lifts}
The next lemma says that any path (of length $>0$) on $G_n$ can be lifted to $G_{n+1}$
(where it becomes shorter if its length is finite).  Note however that the edges of the lift 
are not determined;
only the vertices are determined, because the edges of $p$ do not appear in the definition 
of $\pi_{n+1,n}(p)$.  The ambiguity in the lifting process will play an important role later.

\begin{lem}[Path lifting]
The map $\pi_{m,n}$ is surjective.  The image of $\pi_n|_{S_\infty}$ is $\Path_\infty(G_n)$.
\end{lem}

\begin{proof}
For the first part, it suffices to show that $\pi_{n+1,n}$ is surjective, as $\pi_{m,n}$ is a composition
of maps of this form.  If $p\in\Path(G_n)$ then each edge of $p$ is a vertex of $G_{n+1}$.
We only need to show that if $e,e'\in EG_n=VG_{n+1}$ with $t(e)=h(e')$ then there is an edge 
$f\in EG_{n+1}=S_{n+2}$ with $h(f)=e$ and $t(f)=e'$.  Extend the function $e:[n+1]\to[n+1]$ to 
$\overline{e}:[n+2]\to\R$ by defining $e(n+2)$ such that $e'=\Order(e|_{[2,n+2]})$.  Then 
$f=\Order(e)\in S_{n+2}$ is the desired edge.

For the second part, if $p\in\Path_\infty(G_n)$ then set $p=p_0$ and for each $i>0$
let $p_i\in\Path_\infty(G_{n+i})$ be a lift of $p_{i-1}$.  Then for $m\geq n$ let $\s_m$ be 
the initial vertex of $p_{m-n}$, and note that $\s=(\s_1,\s_2,\dots)\in S_\infty\cap(\pi_n)^{-1}(p)$.
\end{proof}

\begin{example}
Consider the infinite path $p\in \Path_\infty(G_3)$ 
that begins at the vertex $(1 2)$ and
traverses the edges $(1 3 2)$ followed by $(3 1 2)$ repeatedly.
Then $p$ projects to the path $q=\pi_2(p)\in \Path_\infty(G_2)$ which traverses the loop 
labeled $(1 2)$ and then the loop labeled $(2 1)$ and then repeats.  
There are infinitely many paths other than $p$ in $\pi_3((\pi_2)^{-1}(q))$, since the vertices must alternate between $(1 2)$ and $(2 1)$ but there are two choices for each edge.  
By contrast, at the next step, $\pi_4((\pi_3)^{-1}(p))$ is the singleton 
consisting of the infinite path on $G_4$ that starts at the vertex $(3 1 2)$ and traverses the 
edges $(1 4 2 3)$ and $(4 1 3 2)$ repeatedly.  
In fact $(\pi_3)^{-1}(p)\subseteq \Sym_\infty$ is already a singleton, being the compatible 
sequence $(\s_1,\s_2,\dots)$ where $\s_n$ is the permutation $(1,n,2,n-1,\dots)$.  
\end{example}

\section{The poset of a path}

Given a path $p=(v_0,e_1,v_1,\ldots,v_\ell)$ on $G_n$, consider the
set 
$$\overline{Q}_p=\left(\{v\}\times [0,\ell]\times [n]\right) \cup \left(\{e\}\times [1,\ell]\times [n+1]\right).$$
This set is (in 1-1 correspondence with) the disjoint union of the domains of all the permutations
$v_i$ and $e_i$.  They are ``patched together" by the equivalence $\sim$ generated by
\begin{enumerate}
\item[(i)]
$(v,a,c)\sim(e,a+1,c)$
\item[(ii)]
$(v,a,c)\sim(e,a,c+1)$.
\end{enumerate}
The equivalence class of $(v,a,c)$ in $Q_p=\overline{Q}_p/\sim$ will be denoted by $x_{a+c}(p),$
or $x_{a+c}$ if the path $p$ is understood; note that
(a) this is well-defined and (b) every element of $Q_p$ is equal to $x_i$ for some 
$1\leq i \leq \ell+n:=m$.  By (a), if $x_i=x_j$ for $1\leq i,j \leq m$ then $i=j$, and so 
$Q_p=\{x_1,\dots,x_m\}$.

The set $Q_p$ is easy to visualize, but let us first define a partial ordering on it.
  
Consider the relation $\leq$ on $Q_p$ generated by 
\begin{enumerate}
\item[(iii)]
$[(v,a,c)]\leq [(v,a,d)]$ if $v_a(c)\leq v_a(d)$
\item[(iv)]
$[(e,a,c)]\leq [(e,a,d)]$ if $e_a(c)\leq e_a(d)$
\end{enumerate}
and extended by transitivity.

We will show in a moment that $\leq$ is a partial ordering on $Q_p$.
The point of $Q_p$ is to keep track of all order relationships which 
necessarily hold among $\s(1),\ldots,\s(m)$, if $\s$ is a permutation in $\pi_n^{-1}(p)$.

\begin{example}
Consider the path $p$ of length 5 in $G_3$ with edges
$e_1=(2134), e_2=(1342), e_3=(2314), e_4=(3241), e_5=(2314)$.  
This is a loop based at $(213)$.  Attempts to construct real numbers $z_1,\dots,z_8$ 
such that $\Order(z_i,z_{i+1},z_{i+2},z_{i+3})=e_i$ quickly lead one to draw pictures
like Figure \ref{fig:worms}.  The top picture is a plot of the desired $y$'s.  Note that 
$y_5$ could be perturbed to be larger or smaller than $y_1$, and similarly for
$y_7$ and $y_2$.  The dotted lines indicate the duration of the influence of $y_i$ on 
future $y_j$'s.  This information is abstracted in the middle picture, in which the dots 
are the elements of $\overline{Q_p}$ and equivalent elements are joined by an arc.  
Each arc is an element of $Q_p$.  The bottom picture shows the partial ordering:
an edge pointing from $x_i$ to $x_j$ indicates that $x_i\leq x_j$.
\end{example}

\begin{figure}[ht]
  \psfrag{x}{$z_i$}
  \psfrag{1}{$1$}
  \psfrag{2}{$2$}
  \psfrag{3}{$3$}
  \psfrag{4}{$4$}
  \psfrag{5}{$5$}
  \psfrag{6}{$6$}
  \psfrag{7}{$7$}
  \psfrag{8}{$8$}
  \psfrag{i}{$i$}
  \psfrag{y}{$(2 1 3 4)$}
  \psfrag{v0}{$v_0$}
  \psfrag{v1}{$v_1$}
  \psfrag{v2}{$v_2$}
  \psfrag{v3}{$v_3$}
  \psfrag{v4}{$v_4$}
  \psfrag{v5}{$v_5$}
  \psfrag{e1}{$e_1$}
  \psfrag{e2}{$e_2$}
  \psfrag{e3}{$e_3$}
  \psfrag{e4}{$e_4$}
  \psfrag{e5}{$e_5$}
  \psfrag{x1}{$x_1$}
  \psfrag{x2}{$x_2$}
  \psfrag{x3}{$x_3$}
  \psfrag{x4}{$x_4$}
  \psfrag{x5}{$x_5$}
  \psfrag{x6}{$x_6$}
  \psfrag{x7}{$x_7$}
  \psfrag{x8}{$x_8$}
      \centering
  \includegraphics[scale=.46]{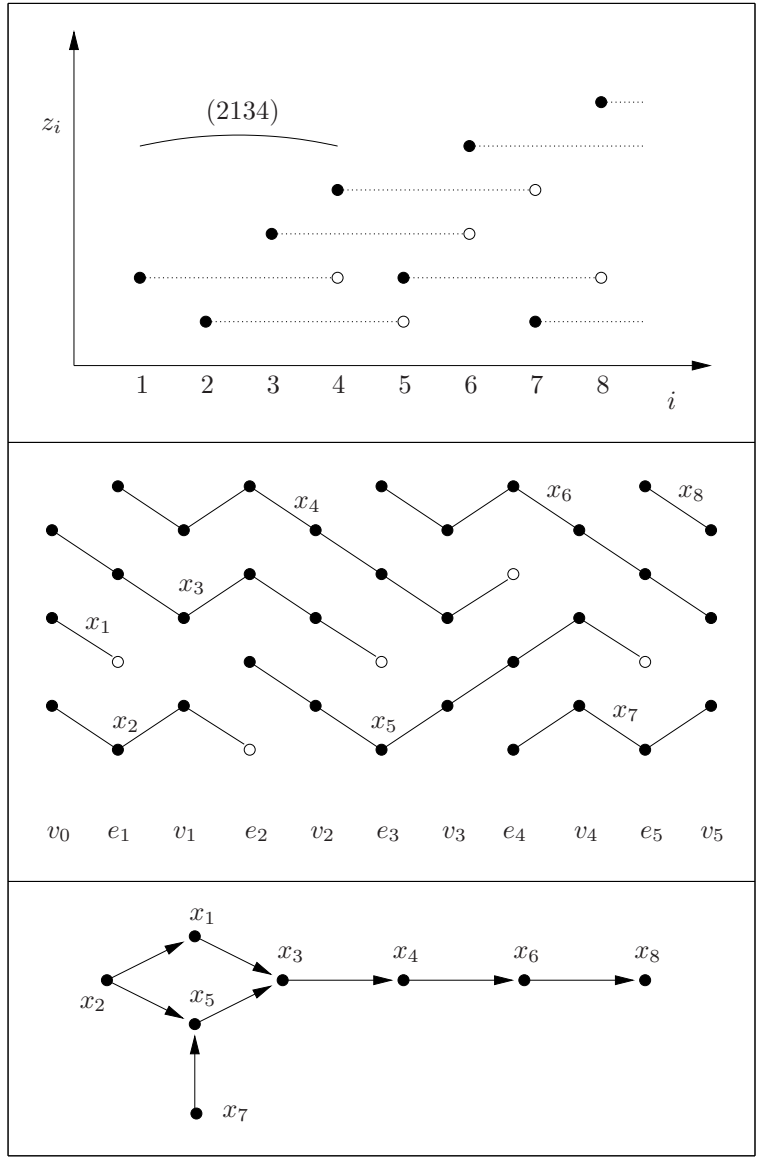}
    \captionsetup{width=\textwidth}
  \caption{The path $p$ is a loop of length 5 in $G_3$ traversing in this order the edges
  $2134, 1342, 2314, 3241, 2314$.  Figure (a) is a plot of a sequence $z_i$ (for $i=1,\dots,8$)
  which maps to $p$ under $\pi_3$.  The elements of $\overline Q_p$, shown in (b) as
  dots, are the intersections of the above plot with vertical lines at 
  $i=3,3.5,4,\dots,8$.  The elements $(v,0,3)$, $(v,0,1)$, $(v,0,2)$ of $\overline Q_p$ make
  up the leftmost column of dots (read from top to bottom).  Equivalent elements are joined 
  by an arc.  Figure (c) depicts the poset $Q_p$ whose elements are the arcs in either of 
  the previous pictures.}
  \label{fig:worms}
\end{figure}

\begin{lem}\label{liftleq}
Let $p$ be a path on $G_n$, and let $\tilde p\in\pi_n^{-1}(p)$.
If $x_i(p)\leq x_j(p)$ then $x_i(\tilde p)\leq x_j(\tilde p)$.
\end{lem}

\begin{proof}
We may assume the length $\ell$ is not zero.  It suffices to prove for 
$\tilde p\in \Path(G_{n+1})$, as we can lift multiple times.

The hypothesis $x_i(p)\leq x_j(p)$ implies there is a sequence
\begin{equation}\label{Q_p-sequence}
x_i\ni(y_0,a_0,c_0)\sim(y_1,a_1,c_1)\leq(y_2,a_2,c_2)\sim\cdots\leq(y_r,a_r,c_r)\in x_j
\end{equation}
in $\overline{Q}_p$, where each step is one of the types (i)-(iv).  In this sequence, 
if one of the inequalities has $y_i=y_{i+1}=v$ then
$a_i=a_{i+1}$ and using (i) and (ii) we can replace $y_i$ and $y_{i+1}$ by $e$ and either
increase $a_i$ and $a_{i+1}$ by 1 (if $a_i\ne\ell$) or increase $c_i$ and $c_{i+1}$ by 1
(if $a_i\ne 0)$.  Since $\rho(e_i)=v_{i-1}$ and $\rho'(e_i)=v_i$, the inequality is preserved
in either case.  Thus we obtain a new sequence \eqref{Q_p-sequence} with each $y_i=e$.

Now note that the elements $(e,a,c)$ of $\overline Q_p$ are in one-one correspondence
with the elements $(v,a,c)$ of $\overline Q_{\tilde p}$.  Thus if we now switch every $e$
to a $v$ and subtract 1 from each $a$, we obtain a sequence in $\overline Q_{\tilde p}$ 
showing $x_i(\tilde p)\leq x_j(\tilde p)$.
\end{proof}

\begin{cor}
The relation $\leq$ is a partial order on $Q_p$.
\end{cor}

\begin{proof}
The relation is reflexive and transitive by definition.  We must show that if $x_i\leq x_j$
and $x_j\leq x_i$ then $i=j$.  Lift $p$ to a path $v_0=\s\in\Path_0(G_m)=S_m$ (where $m=\ell+n$).
There are no equivalences in $\overline{Q}_\s$, and $x_i\leq x_j$ in the poset $Q_\s$ if and only
if $\s(i)\leq\s(j)$.  By Lemma \ref{liftleq}, $x_i\leq x_j$ and $x_j\leq x_i$ in $Q_\s$, so $i=j$.
\end{proof}

\begin{rem}\label{totalorder}
Note that for any $1\leq i \leq \ell$, the elements $x_i,x_{i+1},\dots,x_{i+n}$ are totally ordered in $Q_p$.
\end{rem}

Let $\psi(v)=0$ and $\psi(e)=-1/2$.  If $x_i\leq x_j$ in $Q_p$ then there is a sequence
\eqref{Q_p-sequence} with each $(y_k,a_k,c_k)\in\overline{Q}_p$.  Call such a 
sequence \emph{monotonic} if
the function $\psi(y_k)+a_k$ is monotonic in $k$.

\begin{lem}\label{monotonic}
If $x_i\leq x_j$ in $Q_p$ then there is a monotonic sequence of the form 
\eqref{Q_p-sequence}.  
\end{lem}

\begin{proof}
Choose a sequence of the form \eqref{Q_p-sequence}; one exists by definition.  Note that
only rules (1) and (2) change $\psi$, and that $\psi$ increases (by $1/2$) if either
of these rules is applied by replacing the left side with the right.  Suppose the given 
sequence is not monotonic.  Specifically suppose that $\psi$ increases and later decreases;
the other case is virtually identical.  Choosing an innermost such backtrack,
we find a subsequence of one of the following two forms:
\begin{enumerate}
\item[(i)] 
$(v,a,c)\sim(e,a+1,c)\leq(e,a+1,d)\sim(v,a,d)$
\item[(ii)]
$(e,a,c+1)\sim(v,a,c)\leq(v,a,d)\sim(e,a,d+1)$
\end{enumerate}

In case (i), we have $e_{a+1}(c)\leq e_{a+1}(d)$.  But $v_a=\rho(e_{a+1})$ so we
also have $v_a(c)\leq v_a(d)$.  Thus we can delete the middle two terms of (i) and 
eliminate the backtracking.

Case (ii) is similar:  $v_a(c)\leq v_a(d)$ but this time $v_a=\rho'(e_a)$.  Now it follows
that $e_a(c+1)\leq e_a(d+1)$ so again we can eliminate the backtracking.
\end{proof}

Referring again to Figure \ref{fig:worms}, Lemma \ref{monotonic} says that it is very
easy to determine whether $x_i\leq x_j$.  If $i<j$, one just sees whether it is possible to 
connect the right endpoint of $x_i$ to any point above $x_j$ with a path that
passes the vertical line test.  If not, then $x_i\leq x_j$.

\begin{lem}\label{concat}
Let $p,q$ be paths on $G_n$ of lengths $\ell,\ell'$ such that the concatenation
$pq$ is defined.  For $1\leq i,j \leq \ell'+n$, if $x_i(q)\leq x_j(q)$ then 
$x_{\ell+i}(pq)\leq x_{\ell+j}(pq)$.
\end{lem}

\begin{proof}
Choose a monotone sequence of the form \eqref{Q_p-sequence}, and add
$\ell$ to the second coordinate of each term.  The new sequence proves the
result.
\end{proof}

\begin{rem}
Let $p$ be a path of length $\ell$ on $G_n$, and let $m=n+\ell$.
A choice of lift $\s\in\pi_{m,n}^{-1}(p)$ amounts to a choice of extension of $\leq$ to a total order on 
$\{x_1,\dots,x_m\}$.  That this can be done is well-known; the process is sometimes
called a ``topological sort."  In particular, for a subset $\{i_1,\dots,i_k\}\subset[m]$ of indices,
if the $x_{i_j}$ are pairwise incomparable in $Q_p$ then for any permutation $\nu\in S_k$
there is an extension of $\leq$ to a total order on $\{x_1,\dots,x_m\}$ satisfying
$x_{i_{\nu(1)}}<\cdots<x_{i_{\nu(k)}}$.  In the terminology of lifts this becomes the following
statement, which bears on the discussion of entropy in a later section.
\end{rem}

\begin{cor}\label{factorialgrowth}
Let $p$ be a path of length $\ell$ on $G_n$, and let $m=\ell+n$.  If the elements 
$\{x_{i_1},\dots,x_{i_k}\}$ 
of $Q_p$ are pairwise incomparable, then for any permutation $\nu\in S_k$ there is 
a lift $\s\in S_m\subset\pi_n^{-1}(p)$ such that $\s(i_{\nu(1)})<\cdots<\s(i_{\nu(k)}).$
In particular $|\pi_n^{-1}(p)\cap S_m | \geq k!.$
\end{cor}

As a special case of this we also note the following.

\begin{cor}\label{driftlesscoord}
Let $p$ be a path of length $\ell$ on $G_n$.  Then $x_i\leq x_j$ if and only if 
$\s(i)\leq\s(j)$ for every $\s\in S_m\cap\pi_n^{-1}(p)$ (where necessarily $m=\ell+n)$.
\end{cor}

\section{Drift}
If $\g$ is a loop of length $\ell<\infty$ on $G_n$, then the elements $x_1,\dots,x_n$ of $Q_p$
are totally ordered, as are the elements $x_{\ell+1},\dots,x_{\ell+n}$, and
if we set $y_i=x_{\ell+i}$, then we have $x_i\leq x_j$ if and only if $y_i\leq y_j$, for all $i,j\in[n]$.
The notion of drift is based on how the $x_i$ compare to the $y_j$, as
measured by the following two functions.
Let $\langle n\rangle$ be the totally ordered set $[n]\cup \{-\infty,\infty\}$ 
(with $-\infty<1$ and $n<\infty$), and for $i\in \langle n\rangle $ define
\begin{align*}
\Max_\g(i)&=
\begin{cases}
j &\mbox{if }x_i\leq y_j, \mbox{ and for } k\in[n], x_i\leq y_k \mbox{ implies } 
y_j\leq y_k \\
\infty&\mbox{if }i=\infty \mbox{ or }i\in[n] \mbox{ and there is no $j$ such that }x_i\leq y_j \\
-\infty&\mbox{if }i=-\infty
\end{cases}
\\
\Min_\g(i)&=
\begin{cases}
j &\mbox{if }x_i\geq y_j, \mbox{ and for } k\in[n], x_i\geq y_k \mbox{ implies } 
y_j\geq y_k\\ 
-\infty&\mbox{if }i=-\infty \mbox{ or }i\in[n] \mbox{ and there is no $j$ such that }x_i\geq y_j \\
\infty&\mbox{if }i=\infty
\end{cases}
\end{align*}

\begin{lem}\label{maxminordpres}
If $x_i\leq x_j$ then $\Max_\g(i)\leq\Max_\g(j)$ and $\Min_\g(i)\leq\Min_\g(j)$.
\end{lem}

\begin{proof}
This is immediate from the definitions, and from the fact that $x_i\leq x_j$ if and only
if $y_i\leq y_j$.
\end{proof}

\begin{lem}\label{funct}
Suppose $p$ and $q$ are finite paths such that $pq$ is a path.  Then 
$\Max_{pq}=\Max_q\Max_p$ and $\Min_{pq}=\Min_q\Min_p$.
\end{lem}

\begin{proof}
We give the verification for $\Max$.  For $\Min$, flip the argument upside down.

Let $\ell$ be the length of $p$ and $\ell'$ the length of $q$.

It is clear that the two functions are equal on $\pm\infty$.  
Let $i\in[n]$, let $j=\Max_p(i)$.  If $j=\infty$ then by Lemma \ref{monotonic} it
is impossible to have $x_i\leq x_{\ell+\ell'+k}$ for any $k\in[n]$, so $\Max_{pq}(i)=\infty$.
We may therefore assume $j\ne\infty$, and let $k=\Max_q(j)$.  If $k=\infty$, then again
by Lemma \ref{monotonic} it is impossible to have $x_i\leq x_{\ell+\ell'+k'}$ for any
$k'\in[n]$, so $\Max_{pq}(i)=\infty$.  Thus we may assume $k\ne\infty$.  We want to show
that $\Max_{pq}(i)=k$.

In the poset $Q_{pq}$, we have $x_i \leq x_{\ell+j} \leq x_{\ell+\ell'+k}$.
Also, if $x_i\leq x_{\ell+\ell'+k'}$ then there is a monotonic 
sequence in $\overline{Q}_{pq}$ showing $x_i\leq x_{\ell+\ell'+k'}$.  This sequence must 
contain a point of the form $(v,\ell,j')$, so $x_i\leq x_{\ell+j'}$.  By definition of $j$,
we have $x_{\ell+j}\leq x_{\ell+j'}$, hence $x_{\ell_j}\leq x_{\ell+\ell'+k'}$.  By definition
of $k$ we now have $x_{\ell+\ell'+k}\leq x_{\ell+\ell'+k'}$.  Thus $k=\Max_{pq}(i)$, as
desired.
\end{proof}

Let $\g$ be a loop of length $\ell$ on $G_n$.
For $i,j\in[n]$ let 
$$\Drift_\g(i,j)=
\begin{cases}
+ & \mbox{if } x_i\leq y_j \\
- & \mbox{if } x_i\geq y_j \\
0 & \mbox{otherwise.}
\end{cases}
$$
We will write $\Drift_\g(i)$ for $\Drift_\g(i,i)$.

\begin{dfn}\label{drift}
A loop $\g$ is \emph{partially driftless} if $\Drift_\g(i)= 0$ for some $i\in[n]$.\\
A loop $\g$ is \emph{driftless} if $\Drift_\g(i)=0$ for all $i\in[n]$.  \\
A loop $\g$ is \emph{totally driftless} if $\Drift_\g(i,j)=0$ for all $i,j\in[n]$.  
\end{dfn}

Thus $\g$ is totally driftless if and only if $\Max_\g(i)=\infty$ and $\Min_\g(i)=-\infty$
for all $i\in[n]$, and there is a similar description of driftless and partially driftless loops.

\begin{example}
The loop $p$ in Figure \ref{fig:worms} is partially driftless.  In $Q_p$, we have
$x_1\leq x_6=y_1$ and $x_3\leq x_8=y_3$, so $\Drift_p(1)=\Drift_p(3)=+$.  
However $x_2$ and $x_7=y_2$ are incomparable, so $\Drift_p(2)=0$.  Note 
that the number $z_6$ is necessarily greater than $z_1$, but $z_7$ 
can be chosen to be greater than or less than $z_2$.
\end{example}

\begin{lem}\label{partial}
Let $\beta$ and $\g$ be (partially driftless) loops based at $v$, with 
$\Drift_\beta(j)=\Drift_\g(j)=0$.  Then $\Drift_{\beta\g}(j)=0$.
\end{lem}

\begin{proof}
Let $\ell$ be the length of $\beta$.  Suppose $x_j(\beta\g)\leq y_j(\beta\g)$.  Then 
there is a monotonic sequence \eqref{Q_p-sequence} proving this.  In the sequence 
there must be a representative of $x_i$ for some $\ell+1\leq i \leq \ell+n$.
Now $x_i=x_{\ell+j}$ or $x_i\leq x_{\ell+j}$ would contradict $\Drift_\beta(j)=0$,
and $x_i\geq x_{\ell+j}$ would contradict $\Drift_\g(j)=0$.  By Remark \ref{totalorder}
these are the only possibilities.  Thus $\Drift_{\beta\g}(j) \ne +$.  Similarly $\Drift_{\beta\g}(j)\ne -$.
\end{proof}

\begin{lem}\label{driftsuck}
Let $\beta$ and $\g$ be loops on $G_n$ based at the vertex $v$.
If $\g$ is totally driftless then $\beta\g$ is totally driftless.
\end{lem}

\begin{proof}
Suppose not; then without loss of generality there exists $i\in[n]$ with $\Max_{\beta\g}(i)=j<\infty$.
Thus $x_i(\beta\g)\leq y_j(\beta\g)$ and by Lemma \ref{monotonic} there is a monotonic
sequence proving this inequality.  This sequence must contain 
$(v,\ell,k)$ for some $k$, where $\ell$ is the length of $\beta$.
Starting there, the remainder of the sequence 
(in combination with Lemma \ref{concat}) shows that $\Max_\g(k)\leq j<\infty$,
a contradiction.
\end{proof}

\begin{lem}\label{cyclic}
Cyclic permutations of driftless loops are driftless.
\end{lem}

\begin{proof}
Let $\g=(v_0,e_1,\cdots,v_\ell)$ be a driftless loop and let $\g_k$ be a cyclic
permutation of $\g$ starting at $v_k$.  Suppose $x_i(\g_k)\leq x_{\ell+i}(\g_k)$.
Fix a monotonic sequence showing this inequality, and add $\ell$ to the second
coordinate of each element to obtain a new sequence, and concatenate the original
sequence with the new one.  This longer sequence shows $x_i(\g_k^2)\leq x_{2\ell+i}(\g_k^2)$
but it contains a subsequence showing for some $j$ that $\Drift_\g(j)\ne 0$.
\end{proof}

\begin{dfn}
A \emph{face subgraph} of $G_n$ is a subgraph $H$ such that every edge of $H$ is contained
in a loop in $H$.  Equivalently $H$ is a face subgraph if each connected component 
of $H$ is strongly connected.
\end{dfn}

\begin{dfn}
A strongly connected subgraph $H\subseteq G_n$ \emph{drifts} if there exist 
$v\in VH$, $j\in [n]$ and $\e\in\{+,-\}$ such that for every loop $\g$ in $H$ based
at $v$, $\Drift_\g(j)=\e$.  Otherwise $H$ is \emph{driftless}.

A face subgraph $H\subseteq G_n$ \emph{drifts} if any of its connected components
drifts; otherwise $H$ is driftless.
\end{dfn}

\begin{prop}\label{driftlessloop}
Let $H$ be a strongly connected subgraph of $G_n$.  The following are equivalent:
\begin{enumerate}
\item
$H$ is driftless;
\item
there exists a totally driftless loop $\g$ with support contained in $H$;
\item
there exists a totally driftless loop $\g$ with support equal to $H$.
\end{enumerate}
\end{prop}

\begin{proof}
The last two statements are equivalent by Lemma \ref{driftsuck}:
if $\g$ is a totally driftless loop with support contained in $H$, and $\beta$ is any
loop with support equal to $H$, then $\beta\g$ is a totally driftless loop with support
equal to $H$.

Statement (3) easily implies statement (1):  for fixed $v$, $j$, $\e$ let $\g_v$ be
a cyclic permutation of $\g$ which starts at $v$.  By Lemma \ref{cyclic}
$\Drift_{\g_v}(j)=0\ne\e$.

Last, we show (1) implies (2).  
Let $\g_0$ be a loop based at $\s$ and supported in $H$.  Let $i=\s^{-1}(1)$
and $j=\s^{-1}(n)$, so that $x_i\leq x_k \leq x_j$ for all $k\in[n]$.

Suppose $\Max_{\g_0}(i)=k\ne\infty$.
As $H$ is driftless, we may pick a loop 
$\g_1$ based at $v$ and supported in $H$ such that $\Drift_{\g_1}(k)\ne+$. 
$\Max_{\g_0\g_1}(i)>k$.  We can continue this process until we have a loop
$\beta_0$ with $\Max_{\beta_0}(i)=\infty$.

Then, similarly, we concatenate loops on to the end of $\beta_0$ to create
a loop $\beta$ with $\Min_\beta(j)=-\infty$.  Note that $\Max_\beta(i)=\infty$
(by Lemma \ref{funct}).  Now, by Lemma \ref{maxminordpres} $\beta$ is
totally driftless, with support contained in $H$.
\end{proof}

\begin{cor}\label{enlarge}
If $K$ and $H$ are strongly connected, $K\subseteq H$, and $K$ is driftless,
then $H$ is driftless.
\end{cor}

\section{Measure preserving functions}
\label{mainsec}

In this section we analyze the distributions of order patterns arising from 
(almost aperiodic) measure preserving functions
$$\AMP=\{f\in\A \mid \Leb(f^{-1}(S))=\Leb(S) \ \mbox{ for all measurable sets }S\}.$$
Our main theorem is that the image $\bmu_n(\AMP)$ is a union of open faces of a polytope
$P_n\subset \D_n$ of dimension $n!-(n-1)!$, and that there is an easily checkable 
combinatorial criterion for determining whether a particular face of $P_n$ is in the
image.

\begin{rem}\label{othermeasures}
For most of these results it is not essential that $\Leb$ be the measure preserved
by $f$.  That is, given a function $f\in\A$ one could choose an invariant
measure $\lambda$ and proceed with this section, everywhere replacing $\Leb$
with $\lambda$.  For some steps it may be necessary to assume $\lambda$ has no
atoms.
\end{rem}

We start by observing that Theorem \ref{cantor} would not hold if $\A$ were 
replaced by $\AMP$.  If $\s\in S_n$ let $\d_\s\in\D_n$ denote the distribution whose
value is 1 on $\s$ and 0 elsewhere.

\begin{lem}\label{balayage}
If $J\subseteq I$ has positive measure and $f:J\to J$ is aperiodic and measure preserving
then both $J_+=\{x\in J \mid f(x)>x\}$ and $J_-=\{x\in J \mid f(x)<x \}$ have positive measure.

In particular, there is no $f\in\AMP$ such that $\bmu_2(f)=\d_{(1 2)}$ or $\d_{(2 1)}$.
\end{lem}

\begin{proof}
Suppose $\Leb(J_-)=0$, i.e., $f(x)>x$ for almost all $x\in J$.  Then there is some
$\e$ such that $\Leb\{x\mid f(x)-x>\e\}>0$, hence $\int_J f(x)-x >0$.  But $f$ measure preserving
implies $\int_J f(x)-x=0$.  Similarly for $\Leb(J_+)$.
\end{proof}

Note that $\d_{(1 2)}$ and $\d_{(2 1)}$ are in the closure of $\bmu_2(\AMP)$ since 
$\bmu_2(f_\e)$ can be made arbitrarily close to these distributions by choosing 
$f_\e(x)=x+\e \mod 1$.

\subsection*{The flow polytope $P_n$}

Lemma \ref{balayage} notwithstanding, there is a much more serious reason
for the failure of Theorem \ref{cantor} in the measure preserving category.
For $f\in\AMP$, there is an additional set of constraints on
$\bmu(f)$ beyond compatibility of the measures $\bmun(f)$.  Namely,
the order pattern of $(fx,f^2x, \ldots)$ must be distributed 
in the same way as the order pattern of $(x,fx,f^2x,\ldots)$.
More precisely, if $I_\s=\{x\in I \mid \bs_n^f(x)=\s\}$ then 
$\bmu_n(f)(\s)=\Leb(I_\s)=\Leb(f^{-1}(I_\s))=\Leb(\{x \in I\mid \bs_n^f(f(x))=\s\}).$
Thus if $f\in\AMP$ we necessarily have
\begin{equation}\label{eq:extraconditions}
\mu_n(f)(\s)=\sum\limits_{\rho'(\s')=\s} \mu_{n+1}(f)(\s').
\end{equation}
(Recall that $\rho'(\s)=\Order(\s|_{[2,n+1]})$.)

The functions $\rho$ and $\rho'$, now thought of as maps $S_n\to S_{n-1}$,
induce maps $\rho_*,\rho'_*:\D_n\to\D_{n-1}$.  Explicitly, for $\mu\in\D_n$,
\begin{align*}
\rho_*(\mu)(\s)&=\sum\limits_{\s'\in\rho^{-1}(\s)} \mu(\s')\\
\rho'_*(\mu)(\s)&=\sum\limits_{\s'\in\rho'^{-1}(\s)} \mu(\s').
\end{align*}
Thus by \eqref{eq:extraconditions} and compatibility,
$\rho_*(\bmu(f))=\rho'_*(\bmu(f))$ for $f\in\AMP$.

\begin{dfn}
Set $P_n=\{\mu\in\D_n\mid \rho_*(\mu)=\rho'_*(\mu)\}$.  
\end{dfn}

As each condition \eqref{eq:extraconditions} is linear, $P_n$ is a polytope 
contained in the simplex $\D_n$, and $P_n\cap\bdry\D_n=\bdry P_n$.
We have already proved the following lemma.

\begin{lem}
If $f\in\AMP$ then $\bmu_n(f)\in P_n$.
\end{lem}

\begin{example}\label{ex:polytopes}
The polytope $P_2$ is all of $\D_2$; this is a line segment connecting
$\chi_{(1 2)}$ to $\chi_{(2 1)}$.  The preimage of a point 
$a \chi_{(1 2)} + (1-a) \chi_{(2 1)}\in \Int(P_2)$ under the map $\rho_*$ is a
3-dimensional square pyramid with apex $a\chi_{(1 2 3)}+(1-a)\chi_{(3 2 1)}$.
If $0<a<1/2$ the vertices of the square base are $a(\chi_\s +\chi_\t )+(1-2a)\chi_{(3 2 1)}$ where 
$\s\in\{(1 3 2), (2 3 1)\}$ and $\t\in\{(2 1 3),(3 1 2)\}$, whereas if $1/2<a<1$ then the 
vertices are $(1-a)(\chi_\s+\chi_\t) + 2a\chi_{(3 2 1)}$ with the same choices
for $\s$ and $\t$.  If $a=1/2$ then the square base is a (2-dimensional) face
of $P_3$; it corresponds to the face subgraph $H\subset G_2$ consisting of
all the edges except the loops $(1 2 3)$ and $(3 2 1)$. 

The entire polytope $P_3$ is 4-dimensional; it resembles a suspension of the (middle)
square pyramid, except that the apex of the pyramid lies on the segment connecting
the suspension points $\chi_{(1 2 3)}$ and $\chi_{(3 2 1)}$, so that $P_3$ has six
vertices rather than seven.
See Figure \ref{fig:polytopes} (in which $\rho_*$ projects vertically).
\end{example}

\begin{figure}[ht]
  \psfrag{1}{{\small$\chi_{123}$}}
  \psfrag{2}{{\small$\chi_{321}$}}
  \psfrag{3}{{\small$\chi_{12}$}}
  \psfrag{4}{{\small$\chi_{21}$}}
  \psfrag{5}{{\small$\frac 12( \chi_{132}+\chi_{213})$}}
  \psfrag{6}{{\small$\frac 12( \chi_{231}+\chi_{213})$}}
    \centering
  \includegraphics[width=3.5in]{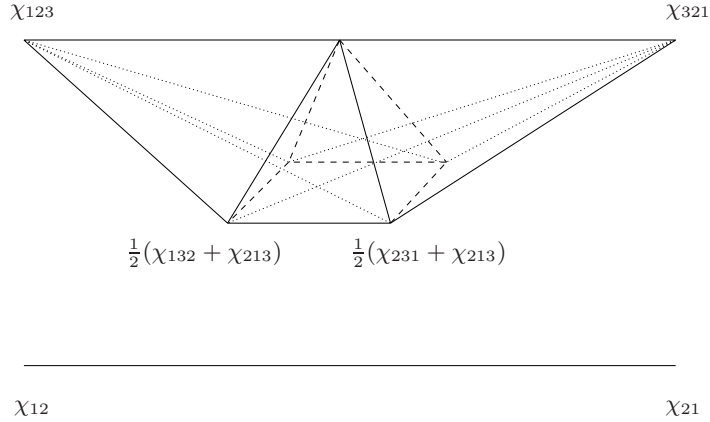}
  \captionsetup{width=.75\textwidth}
  \caption{The polytopes $P_2$ (below) and $P_3$ (above).  Fibers of the (vertical) projection
  are 3-dimensional square pyramids; the preimage of the midpoint of $P_2$ is shown.
  The whole polytope $P_3$ is the join of an interval and a square.}
  \label{fig:polytopes}
\end{figure}

\subsection*{Dictionary between $P_n$ and $G_{n-1}$}
Before we get to the main theorem we establish several connections between
$P_n$ and $G_{n-1}$.

An \emph{edge weighting} on a digraph $G$ is a map $\phi:EG\to[0,1]$ such that 
$\sum \phi(e)=1$.  A \emph{flow} on $G$ is an edge weighting $\phi$ such that for every
$v\in VG$,
$$\sum\limits_{\{e\mid h(e)=v\}} \phi(e)=\sum\limits_{\{e\mid t(e)=v\}} \phi(e).$$
Note that the set of all edge weightings on $G_{n-1}$ is exactly $\D_n$, and
the set of all flows on $G_{n-1}$ is exactly $P_n$.

A flow supported on an embedded loop in $G_{n-1}$ is a vertex of $P_n$.
The set of all flows supported on a face subgraph $H\subset G_{n-1}$ is a face
$F_H$ of $P_n$.  The assignment $H\mapsto F_H$ is an inclusion-preserving
bijection between the set of face subgraphs of $G_{n-1}$ and the set of faces
of $P_n$.  The dimension of $F_H$ is one less than the rank of the first homology 
of $H$.  In particular, if $H=G_{n-1}$ then $F_H=P_n$ has dimension $n!-(n-1)!$.

If two face subgraphs $H,K\subset G_{n-1}$ are disjoint, then 
$F_{H\cup K}=F_H \ast F_K$ where $\ast$ denotes the join.

\begin{example}
By counting the face subgraphs of various ranks in, say, $G_2$, one determines 
the number and structure of faces of $P_3$ of each dimension.  It is instructive
to compare this with the earlier description of $P_3$ given in Example \ref{ex:polytopes}.
\end{example}

\begin{rem}
The dimension of $\D_n$ is $n!-1$, and the conditions \eqref{eq:extraconditions}
impose $(n-1)!$ additional linear constraints.  These constraints are obviously 
independent, since their sum is zero; the fact that $P_n$ has dimension $n!-(n-1)!$
shows that the constraints are otherwise linearly independent.
\end{rem}

\subsection*{Realizable faces}

Here is our main theorem, which we prove after a sequence of lemmas.

\begin{thm}\label{main}
\begin{enumerate}
\item  
The set $\bmu_n(\AMP)$ is a union of open faces of $P_n$.\\
\item
Let $F$ be a face of $P_n$ and let $H$ be the corresponding face subgraph of $G_{n-1}$,
so that $F=F_H$.  Then $\Int(F)\subset\bmu_n(\AMP)$ if and only if $H$ is driftless.
\item
The closure of $\bmu_n(\AMP)$ is $P_n$.
\end{enumerate}
\end{thm}


\begin{example}
The set $\bmu_2(\AMP)$ is equal to the interior of $P_2$.  The set
$\bmu_3(\AMP)$ consists of $\Int(P_3)$ (which is 4-dimensional) together
with all six of its open 3-dimensional facets, nine of its thirteen open 2-dimensional
faces (including the square face), and two of its thirteen open edges.  None of the 
six vertices of $P_3$ is in $\bmu_3(\AMP)$.

There are sometimes vertices of $P_n$ in $\bmu_n(\AMP)$.  For example the 
embedded loop in $G_4$ with edges $(23451),$ $(34512),$ $(45132),$ $(41325),$ 
$(13254),$ $(31542),$
$(15423),$ $(54123),$ $(51234)$ is driftless, as is easily seen by computing its poset $Q$.  
Hence by Theorem \ref{main} the corresponding vertex of $P_5$ is realizable.
\end{example}

\begin{lem}\label{functions}
Let $\g$ be a driftless loop in $G_{n-1}$.  Then there is
$f\in\AMP$ such that $\mu_n(f)$ equals the counting measure induced on $EG_{n-1}$ 
by $\g$.  In particular $\mu_n(f)$ is in the interior of the face $F_H$, where
$H$ is the (edge) support of $\g$.
\end{lem}

\begin{proof}
Lift $\g$ to a permutation $\s\in S_m$.  Let $\phi$ be a measure preserving
ergodic function $I\to I$.

We build the \emph{permutation function} corresponding to $\s$:
for $\s\in\Sym_\ell$, set 
$$
\overline f_\s (x) = x+ \frac {\s(i+2)-\s(i+1)} {\ell} \quad \mbox{where }\ i=\lfloor nx \rfloor.
$$
Finally, let $f_\s$ equal $\overline f_\s$ composed with a scaled down version of $\phi$ 
on the interval $[0,1/m]$.  Now $f_\s$ has the desired property.
\end{proof}

\begin{lem}[\emph{Balayage}]\label{driftlessface}
Let $H$ be a connected face subgraph of $G_n$.  Then $H$ is driftless 
if and only if $\Int(F_H) \cap \mu_n(\AMP)\ne\emptyset$.
\end{lem}

\begin{proof}
Assume $H$ driftless.  By Lemma \ref{driftlessloop} there
is a totally driftless loop $\gamma$ with support $H$.  By Lemma \ref{functions} 
there is $f\in\AMP$ with $\bmu_n(f)\in\Int(F_H)$.

Conversely, assume $H$ drifts.  Let $f\in\AMP$, and suppose that 
$\bmun(f)\in\Int(F_H)$.  Using the drift, we will construct from $f$ a positive measure 
subset of $I$ and a measure preserving function $g$ such that either $g(x)>x$ for all
$x$ or $g(x)<x$ for all $x$.  This will contradict Lemma \ref{balayage}.

Let $J=\{x\in I_{ap}\mid \forall u\in VG_n, |\{i\mid\bs_n^f(f^i(x))=u\}|\in\{0,\infty\}\}$.
Note that $\Leb(J)=1$, since $J\subseteq f^{-1}(J)$ and $I_{ap}=\cup_if^{-i}(J)$.  
Let $v\in VH,j\in[n],\e\in\{+,-\}$ be as asserted in the definition of drift.
Set $J_v=\{x\in J\mid \bs_n^f(x)=v\}$ and 
$J_{v,j}=J\cap f^{(j-1)} (J_v).$
Note $\Leb(J_{v,j})\geq\Leb(J_v)$ are positive by hypothesis.
For $x\in J_v$ let $i(x)$ be the smallest $j>0$ such that $f^j(x)\in J_v.$
For $y\in J_{v,j}$ we write $y=f^{j-1}(x)$ with $x\in J_v$, and now define 
$g:J_{v,j}\to J_{v,j}$ by $g(y)=f^{i(x)}(y).$

Note that $g$ is measure preserving.  To see this consider
$A\subseteq J_{v,j}$ measurable and write 
$B_r=f^{-r}(A)-\cup_{i\in[0,r-1]}f^{-i}(J_{v,j})\subseteq \cup_if^{-i}(J_{v,j})-\cup_{i<n}f^{-i}(J_{v,j})$ 
a sequence with measure decreasing to $0.$  
Write $A_r=B_r\cap J_{v,j}.$  
Note that $g^{-1}(A)=\cup_rA_r$ is a disjoint decomposition and for every 
$n$ there is $\Leb(A)=\Leb(\cup_{r<n}A_r)+\Leb(B_n)$ so that $\Leb(A)=\Leb(g^{-1}(A)).$

Now if $\e=+$, then $g(y)>y$ for all $y\in J_{v,j}$, and if $\e=-$, then
$g(y)<y$ for all $y\in J_{v,j}$.  Either case contradicts Lemma \ref{balayage}.
\end{proof}

\begin{lem}\label{disconnected}
For any face subgraph $H$ of $G_n$, $\Int(F_H)\cap \bmu_n(\AMP)\ne\emptyset$ 
if and only if $\Int(F_K)\cap \bmu_n(\AMP)\ne\emptyset$ 
for every connected component $K$ of $H$.
\end{lem}

\begin{proof}
Suppose $f\in\AMP$ and $\bmu_n(f)=\mu\in\Int(F_H)$.
Let $K$ be a connected component of $H$, and let 
$I_K=\{x\in I \mid \bs_n^f(x)\in VK\}$.
Note that $\Leb(I_K)\ne 0$; defining $g$ to be a scaled up version
of $f|_{I_K}$ so that $g:I\to I$, we have $\bmu_n(g)\in\Int(F_K)$.

The converse implication follows from the convexity of $\bmu_n(\AMP)$.
\end{proof}

We now prove Theorem \ref{main}.

\begin{proof}[Proof of main theorem]
To prove (1), let $\Int(F_H)$ denote the open face $F_H$.
We will show that if $\Int(F_H)\cap\bmu_n(\AMP)$ is nonempty then
for each vertex of $F_H$ there are points of $\Int(F_H)\cap\bmu_n(\AMP)$
arbitrarily close to $v$.  By convexity of $\mu_n(\AMP)$ it follows that
$\Int(F_H)\subset\bmu_n(\AMP)$, thus proving (1).

Suppose $\Int(F_H)\cap\bmu_n(\AMP)$ is nonempty.  If $H$ is
connected, then by Lemma \ref{driftlessface}
$H$ is driftless, and by Lemma \ref{driftlessloop} there is a totally driftless loop
$\g$ with support $H$.
Let $v$ be a vertex of $F_H$ and let $\beta$ be an embedded loop in $H$ such that 
$v=F_\beta$.  By Lemma \ref{concat} the loop $\beta^N\g$ is driftless, so by 
Lemma \ref{functions} there is $f\in\AMP$ with $\bmu_n(f)$ equal to the counting
measure on the loop $\beta^N\g$.  As $N$ grows this sequence of measures
approaches $v$.

If $H$ is not connected, then by Lemma \ref{disconnected}, for each connected
component $K$ of $H$ there is $f_K\in\AMP$ with $\bmu(f_K)\in\Int(F_K)$.  We apply
the argument from the previous paragraph to each face $F_K$, obtaining points
of $\bmu_n(\AMP)$ close to the vertices of $F_K$.  As each vertex of $F_H$ is
a vertex of one of the $F_K$'s, we are done.

As for (2), by (1) we know that $\Int(H)\subset\bmu_n(\AMP)$ if and only if
$\Int(H)\cap\bmu_n(\AMP)\ne\emptyset$.  If $H$ is connected, Lemma \ref{driftlessface}
finishes it.  If $H$ is not connected, then for any connected component $K$ of $H$
we have $\Int(F_K)\cap\bmu_n(\AMP)\ne\emptyset$ if and only if $K$ is driftless.
So by Lemma \ref{disconnected}, $\Int(H)\subset\bmu_n(\AMP)$ if and only if each
$K$ is driftless, i.e., if and only if $H$ is driftless.

To prove (3), it suffices to show that $\Int(P_n)\subset \bmu_n(\AMP)$.  This is
easy:  as $\bmu_n(\AMP)$ is not empty, there must exist a (connected) driftless
face subgraph.  By Corollary \ref{enlarge}, the whole graph $G_{n-1}$ is driftless.  
Since $\Int(P_n)=F_{G_{n-1}}$, the result is implied by (2).
\end{proof}

\begin{cor}\label{uniformdist}
For each $n$, there exists $f\in\AMP$ such that $\bmu_n(f)$ is uniform on $S_n$.
\end{cor}

\begin{rem}
We have answered Question 1 for $\C=\AMP$.  However Question 2 remains open.
In particular, we do not know if there is $f\in\AMP$ such that $\bmu_n(f)$ is uniform 
for all $n$.  See Section 8.
\end{rem}
%

\section{Entropy and finite exclusion type}

In this section we change our focus from the distribution $\bmu_n(f)$ to a coarser
statistic, namely the number of permutations of length $n$ realized by $f$.
We relate two notions about a continuous piecewise monotone function $f$:  
finite entropy and finite exclusion type.  
The basic idea is that these two concepts imply opposite things for the number of length $N$
permutations realized by iterates of $f$ as $N$ gets large.  Roughly speaking, finite entropy 
implies that the number of permutations realized by $f$ grows (at most) exponentially 
in the length.  On the other hand, finite exclusion type means 
that the only restrictions on the permutations realized by $f$ are given by looking at 
permutations of a fixed finite length.  Often, this will imply that the number of 
realizable permutations in $S_N$ grows super-exponentially in $N$.

Define $\bs_n (f)$ to be the image of $\bs_n^f$ in $S_n$.

\subsection*{Continuous functions and entropy}

For (piecewise) continuous functions, several classical definitions of the topological 
entropy $h(f)$ are possible.  The reader is referred to \cite{walters} for details.
A new notion of entropy called \emph{topological permutation entropy} has been 
studied recently by several people;
the following combines Theorem 1 of \cite{bkp} with Theorem 2.1 of \cite{m}.

\begin{thm}\label{entropy}
If $f:I\rightarrow I$ is piecewise continuous and piecewise monotone then 
$h(f)=\lim_{n\rightarrow\infty}\frac{1}{n-1}\log(|\bs_n(f)|)$ and $h(f)$ is finite.
\end{thm}

\subsection*{Finite exclusion type}

\begin{dfn}
A function $f\in\A$ has {\em exclusion type $n$} if
there exists $H\subseteq G_n$ such that $\bs_m (f) = \pi_{m,n}^{-1}(\Path_{m-n} H)$
for all $m\geq n$ and \emph{finite exclusion type} if it has exclusion type $n$
for some $n$.
\end{dfn}

Note that this says not only that every path in $G_n$ realized by $f$ is supported
on $H$, but also that every lift of every path supported on $H$ is realized by $f$.
A condition equivalent to finite exclusion type is that there are finitely many
\emph{basic forbidden patterns} for $f$, in the language of \cite{sergi2}.  This means
that there are finitely many permutations $\s_1,\dots,\s_k$ such that any permutation 
$\s$ (of any length $m$) either occurs as $\bs_m^f(x)$ for some $x$ or else satisfies 
$\Order(\s|_J)=\s_i$ for some interval $J\subset[m]$ and some $i$.  Elizalde has 
proposed the problem of characterizing those functions which have finite exclusion
type.  We will give a necessary condition.

\begin{thm}\label{partiallydriftlessthm}
Suppose $f$ has finite exclusion type $n$, and let $H\subset G_n$ be the associated
subgraph.  If $H$ contains a partially driftless loop then $|\bs_N(f)|$ grows super-exponentially;
i.e., for any $c\in\R$, we have 
$$|\bs_N(f)| > c^N \qquad \mbox{ for sufficiently large }N.$$
\end{thm}

\begin{proof}
Let $\g$ be a loop on $H$ with $\Drift_\g(j)=0$ for some particular $j\in[n]$.
Let $\ell$ be the length of $\g$, and set $m_k=k\ell+n$.  By the hypothesis of
finite exclusion type we have 
$$\s_{m_k}(f)=\pi_{m_k,n}^{-1}(\Path_{k\ell}(H))\supset \pi_{m_k,n}^{-1}(\g^k).$$
Now since $\Drift_\g(j)=0$, the $i$ elements $x_j,x_{\ell+j},\dots,x_{(k-1)\ell+j}$ 
of the poset $Q_{\g^k}$ are pairwise incomparable, by Lemma \ref{partial}.
Thus the number of lifts of $\g^k$ to $S_{m_k}$ is at least $k!$, by Corollary 
\ref{factorialgrowth}.  Therefore $|\bs_{k\ell+n}(f)|\geq k!$ for all $k$ and the 
result follows.
\end{proof}

\begin{rem} 
Elizalde and Liu \cite{sergi3} have shown that there is no piecewise monotonic 
function $f:I\to I$ of finite exclusion type with associated graph $H\subset G_2$ 
where $EH=\{(123), (321), (213), (312)\}$.  This does not follow from the preceding
theorem, as this $H$ contains no partially driftless loop.
\end{rem}

For a given function $f$, denote by $H_n(f)$ the subgraph of $G_n$ with edge
set $\bs_{n+1}(f)$.

\begin{thm}\label{ergodic}
If $f:I\to I$ is ergodic then for every $n$, $H_n(f)$ contains a partially driftless loop.  
\end{thm}
\begin{proof}  
Consider the graph $H$ with vertex set $VH=VG_n$, edge set 
$EH=\{(x,N)\in I\times \N\mid N>n, \forall 1\leq m\leq N+n, d(f^m(x),x)\geq d(f^N(x),x)\}$ 
and head and tail maps the restrictions to the initial and final segments of $\bs_N^f(x)$.  
Note that any directed cycle in $H$ yields a partially driftless loop in $H_n(f)$ and that $H$ 
has finitely many vertices so it suffices to construct an infinite path in $H$.  

Consider $J=\{x\in I\mid\forall y\in I, \epsilon>0, 
\lim_{n\rightarrow \infty}\frac{|\{m<n|d(f^m(x),y)<\epsilon\}|}{n}\in(\frac{\epsilon}{2},4\epsilon)\}$.
By the compactness of $I$ and ergodicity of $f$, $\Leb(J)=\Leb(I)=1$.  
Since $f(J)\subseteq J$ there will be an infinite path in $H$ if $J\subseteq\pi_1(EH)$;
this is shown next.  For any $x\in J$ choose $\epsilon<\frac{1}{4n}$ so that if $1\leq m\leq n$ 
then $d(f^m(x),x)>\epsilon$.  Choose $r>n$ with $d(f^r(x),x)<\epsilon$.  Choose $N>r>n$ with 
$d(f^m(x),x)\geq d(f^N(x),x)$ for every $1\leq m\leq N+n$ (so that $(x,N)\in EH$). 
Such an $N$ exists since there is always eventually another sequence of length $n$ 
avoiding the $\epsilon$ ball around $x$.  
\end{proof}

\begin{thm}\label{periodic}
If $f:I\to I$ is piecewise continuous and if $x_0$ is a periodic point of period $p>n$
such that $f$ is continuous at every iterate of $x_0$, then $H_n(f)$ 
contains a partially driftless loop.  
\end{thm}

\begin{proof}
Using continuity, choose
$\e>0$ so that for any $x\in I_{ap}$ within $\e$ of $x_0$, the balls $B_\e(f^i(x))$ are 
pairwise disjoint for $i=0,\ldots,p-1$ and the iterates satisfy $|f^i(x)-f^{p+i}(x)|<\e$ 
for $0\leq i \leq n-1$.  Then the image in 
$G_n$ of $\bs_{p+n}^f(x)$ is a partially driftless loop.  
\end{proof}

\begin{cor}\label{entropyresult}
If $f:I\to I$ is piecewise continuous and piecewise monotonic and either 
\begin{itemize}
\item
$f$ is ergodic on a subinterval of $I$, or 
\item
$f$ has arbitrarily large finite orbits on which it is continuous,
\end{itemize}
then $f$ does not have finite exclusion type.  
\end{cor}

Recall that by Sarkovskii's Theorem \cite{sarkovskii}, a continuous function has points 
of arbitrarily large period as long as there is a periodic point whose period is not 
a power of 2.

\section{Proof of Theorem \ref{cantor}}
\label{cantorproof}

We now give the promised proof of Theorem \ref{cantor}.  Given
$\mu=(\mu_1,\mu_2,\ldots)$ we will construct $f$ with $\bmu(f)=\mu$.
Our construction will involve several layers of Cantor sets, and the
resulting functions will be nowhere near continuous or measure preserving.

Recall that $\rho(\s)=\Order(\s|_{[n-1]})$ if $\s\in S_n$.

\begin{lem}  \label{l:I's}  Given $\mu \in \D_{\infty}$, there exist intervals 
$\{ I_{\s} \subset  (\frac{1}{4}, \frac{3}{4}]\}_{\s \in \cup_n  S_n}$, open at the 
left endpoint and closed at the right endpoint, with the properties that:
\begin{enumerate}
\item[(i)] $I_{\s} \cap  I_{\tau} = \emptyset $ for all $\s \not= \tau \in S_n$,
\item[(ii)] $I_ \s \subset I _{\rho (\s) }$,
\item[(iii)] $\Leb(I_{\s} )= \frac{1}{2} \mu _n  ( \s )$ for all $\s \in S_n$,
\item[(iv)] for each $n$, $\cup_{\s\in S_n} I_\s = (\frac 14,\frac 34]$.
\end{enumerate}
\end{lem}

\begin{proof}
We define the $I_{\s}$ inductively as follows. First set $I_{(1)}=(\frac{1}{4}, \frac{3}{4}]$.  
Now let $n>1$ and assume that intervals  $I_{\tau}$ have been constructed for all 
$\tau\in S_{n-1}$.  Since $\mu_{n-1}(\tau)=\sum_{\s} \mu_n(\s)$, summed 
over all $\s\in S_n$ such that $\rho(\s)=\tau$, we may subdivide each  
$I_{\tau}$ into half-open intervals $I_{\s}$ of length $\frac{1}{2} \mu _n  ( \s )$.
\end{proof}

\begin{lem} \label{l:J's} 
There exist disjoint intervals $\{ J_\s \} _ {\s \in \cup_n S_n}$ such that for 
all compatible sequences $( \s _1, \dots \s _n)$, and any 
$(x_1, \dots , x_n)$ with $x_i \in J_{\s_i}$,  $\Order (x_1, \dots , x_n)= \s _n$.
\end{lem}

\begin{proof}
Again the construction is inductive.  Suppose that the $J_\s$ have been 
constructed for $\s\in \cup_{n=1}^k S_n$, and assume further that gaps of 
positive lengths exist between these intervals and at both endpoints.  Order the 
permutations in $\s\in S_{k+1}$ arbitrarily, and for each such $\s$, let 
$J_\s$ be an arbitrary open interval disjoint from the previously chosen 
intervals and with positive length gaps away from them, subject to the further 
condition that $J_{\s}$ should lie in the correct gap as determined by the 
value of $\s(k+1)$.  
\end{proof}

\begin{proof}[Proof of Theorem \ref{cantor}]
Let $\mu = (\mu _1, \mu_2, \dots ) \in \D _{\infty}$ be given; we will construct
a function $f\in\A$ with $\bmu(f)=\mu$.  The construction proceeds in a sequence
of steps.

	{\it Step 1.}  Let $C$ denote the (usual) Cantor set in $[0,1]$.   By applying 
an order preserving transformation we can assume that the $\{J_{\s}\}$ 
given by Lemma \ref{l:J's} have the additional properties that $J_1=[\frac{1}{4}, 
\frac{3}{4}]$ and  $J_{\s} \subset [\frac{1}{8}, \frac{7}{8}]$ for all permutations 
$\s$.  For each permutation $\s$ choose an order preserving injection 
$\phi _{\s}: C \rightarrow J_{\s}$ with $\Leb(\phi_\s(C))=0$.   
Let $I_{\s}$ be as in Lemma \ref{l:I's}.  
Finally choose $\beta:  [\frac{1}{4}, \frac{3}{4}] \rightarrow C$ to be an order 
preserving bijection.

We define the function $f_1$ on a subset of $[0,1]$ recursively, as follows.  

\begin{itemize}
\item[--] 
First, on $ (\frac{1}{4}, \frac{3}{4}]=I _{(1)}= I_{(12)} \cup I_{(21)}$:
for each $\s\in S_2$, if $x\in I_\s$ then set $f_1(x) = \phi _\s \beta (x)$.
Thus for $\s \in S_2$, we have $f_1 (I_{\s}) \subset J_{\s}$.

\item[--]
Next, assuming $f_1$ is defined on $f_1 ^{i-1} (I_{(1)})$, we define $f_1$ 
on $f_1 ^i(I_{(1)})$ as follows.  Notice that 
$f_1 ^{i}(I_{(1)}) = \sqcup _{\s \in S_{i+2}} f_1^i (I_\s)$.  For all 
$\s \in S_{i+2}$ and for all  $ x \in f_1^i (I_\s)$ define 
$f_1(x) = \phi _{\s} (\phi_{\rho(\s)}^{-1} (x))$.  Thus we have
$f_1 ^i (I_\s) \in \phi _{\s}(C) \subset J_{\s}$.

\end{itemize}

We have now defined every power of $f_1$ on $(\frac{1}{4}, \frac{3}{4}]$; note that the
the domain of $f_1$, which we will call $D$, is $(\frac{1}{4}, \frac{3}{4}] $ union a measure 
zero set.  The purpose of this construction is that for any $\s\in S_n$ and $x\in I_\s$,
we now have $\bs_n^{f_1}(x) =\s$.  We set $f=f_1$ on $D$.

{\it Step $m$, $m\geq 2$.}  Denote by $K_m$ the measure zero set in 
$(0, 2^{-m}] \cup (1- 2^{-m}, 1]$ for which $f(x)$ has already been defined.  
Define $g_m: (0, 2^{-m}] \cup (1- 2^{-m}, 1] \rightarrow (0,1]$ to be the map 

$$g_m(x) = 
\begin{cases}
2^{m-1} x &\mbox{ for } x\in (0, 2^{-m}], \\
1 - 2^{m-1}(1-x) &\mbox{ for } x\in (1- 2^{-m}, 1].
\end{cases}
$$

For all $x \in g_m ^{-1} (D) \setminus K_m$,  define $f(x) =g_m^{-1} (f_1 (g_m(x)))$.  
Note that if $\s\in S_n$ and $x \in g^{-1}_m (I_{\s}) \setminus K_m$, 
we now have
\begin{equation} \label{e:gordersigma}
\bs_n^f(x) = \s.
\end{equation}

After step $m$, the domain of $f$ includes the interval $(2^{-m},1-2^{-m}]$,
so the iterative process defines $f$ on $(0,1)$.   It remains to show that 
$\bmu_n(f)=\mu _n$ for all $n$.  

For $\s \in S_n$,  define 
$\overline{I}_{\s}= I_\s \cup ( \cup_{i \geq 2}\  g^{-1}_i ( I_\s))$.  
By \eqref{e:gordersigma} we have $\bs_n^f(x)=\s$
for all $x\in \overline{I}_{\s} \setminus (\cup_m K_m)$.
Thus $\bmu_n(f) (\s) = | \overline{I}_{\s} \setminus (\cup_m K_m) |$.  Since 
$(\cup _m K_m)$ is of measure zero, 
$$ | \overline{I}_{\s} \setminus (\cup_m K_m) |= | \overline{I_{\s}}|=
\frac{1}{2} \mu _n(\s) + \sum_{i\geq 2} \  2^{-(i-1)} (\frac{1}{2} \mu_n (\s))= \mu _n (\s) . $$

We conclude that $\bmu_n(f)=\mu_n$ for each $n$, and $\bmu(f)=\mu$.
This completes the proof.
\end{proof}

\section{Open Questions}

Many interesting open questions remain about the relationship between functions and 
their distributions of order paterns.  

\subsection*{Measure preserving functions}
The bulk of the work presented here focused on the 
class of measure preserving functions; however to date we have been unable to answer 
Question 2 for this class.
\begin{question}
What is $\bmu(\AMP)$?
\end{question} 
There is an infinite version of the polytope, $P_\infty$, which consists of compatible sequences
$(\mu_1,\mu_2,\dots)$ with $\mu_i\in P_i$.  We do not know if the ``interior'' of $P_\infty$
is realizable by some $f\in\AMP$ (where the meaning of ``interior'' depends on the topology
on $P_\infty$), or if there is a drift condition for faces.  One concrete question is this:
\begin{question}
Is there $f\in\AMP$ with $\bmu_n(f)$ uniform for all $n$?
\end{question}

Corollary \ref{uniformdist} asserts that such an $f$ exists for any particular $n$,
and of course by Theorem \ref{cantor} there is $f\in\A$ that works for all $n$.
Yet there is no piecewise monotonic $f\in\A$ that works for all $n$,
because such an $f$ would have finite entropy (by \cite{bkp}, or Theorem \ref{entropy})
hence $|\bs_N(f)|$ would grow at most exponentially in $N$.
Note that such a function might be desirable as a random number 
generator, since from the point of view of order patterns, its iterates 
would look perfectly random.

In a somewhat different direction, if $\lambda$ is a reasonably nice 
measure on $I$ then the results of Section \ref{mainsec} hold with $\C=\AMP$ 
replaced by the collection $\C=\A^\lambda$ of functions which preserve $\lambda$.
(See Remark \ref{othermeasures}.)

\begin{question}
Are there measures $\lambda$ for which $\bmu_n(\A^\lambda)\ne\bmu_n(\AMP)$?
\end{question}

\subsection*{Other functions}
Returning to the broader Questions 1 and 2, there are several interesting
classes of functions $\C$ to study, such as (piecewise) continuous functions, 
polynomials, etc.  For example, if $\C=\A^{\rm{pc}}$ is the collection of piecewise
continuous functions, then it is easy to see that the only vertices of $\D_n$
contained in $\bmu_n(\A^{\rm{pc}})$ are $\chi_{(12\cdots n)}$ and 
$\chi_{(n\cdots21)}$.

\begin{question}
Is the closure of $\bmu_n(\A^{\rm{pc}})$ equal to $\D_n$?
\end{question}
\begin{question}
Is there a drift criterion which applies to $\A^{\rm{pc}}$?
\end{question}

Finally, it would be natural to study the extent to which 
the distributions $\bmu_n(f)$ determine $f$, for $f$ in a given class $\C$.

\begin{question}
For $\mu\in\bmu_n(\C)$, what is $\bmu_n^{-1}(\mu) \cap \C$?
\end{question}
\begin{question}
For $\mu=(\mu_1,\mu_2,\dots)\in\bmu(\C)$, what is $\bmu^{-1}(\mu) \cap \C$?
\end{question}

These questions are in a sense converse to Questions 1 and 2.


\begin{thebibliography}{[22]}

\bibitem{bkp}
C.~Bandt, G.~Keller, and B.~Pompe, \emph{Entropy of interval maps via permutations,}
Nonlinearity 15 (2002) 1595--1602.

\bibitem{bp}
C.~Bandt and B.~Pompe, \emph{Permutation entropy:  A natural complexity measure for 
time series,} Phys.~Rev.~Lett.~88 (2002) 174102.

\bibitem{ak}
J.~M.~Amig\'o and M.~Kennel, \emph{Topological permutation entropy,} Physica D 231
(2007) 137--142.

\bibitem{sergi1}
J.~M.~Amig\'o, S.~Elizalde, and M.~Kennel, \emph{Forbidden patterns and shift systems,}
J.~Combin.~Theory Ser.~A 115 (2008) 485--504.

\bibitem{sergi2}
S.~Elizalde, \emph{The number of permutations realized by a shift}, SIAM J.~Discrete 
Math.~23 (2009) 765--786.

\bibitem{sergi3}
S.~Elizalde and Y.~Liu, \emph{On basic forbidden patterns of functions,} arXiv:0909.2277.

\bibitem{walters}
P.~Walters, An Introduction to Ergodic Theory, Springer-Verlag, New York, 1982.

\bibitem{sarkovskii}
O.~Sarkovskii, \emph{Co-existence of cycles of a continuous mapping of a line into itself,}
Ukrain.~Mat.~Z.~16 (1964) 61--71.

\bibitem{m}
M.~Misiurewicz, \emph{Permutations and topological entropy for interval maps,}
Nonlinearity 16 (2003) 971--976.



\end{thebibliography}
\end{document}